\documentclass[11pt]{amsart}
\usepackage[top = 1in, left = 1.25in, right = 1in, bottom = 1in]{geometry}
\usepackage{pdfpages}

\usepackage{amssymb}
\usepackage{amsthm}
\usepackage{amsmath}
\usepackage{etex}
\usepackage{xcolor,colortbl}
\usepackage[all]{xy}
\usepackage{booktabs}
\usepackage{tikz}
\usetikzlibrary{matrix}
\usepackage{array}

\usepackage{amsrefs}
\usepackage{mathrsfs}
\usepackage{verbatim}
\usepackage{graphicx}

\def\be{\begin{equation}}
\def\en{\end{equation}}

\definecolor{darkgreen}{rgb}{.1,.6,0}

\newcommand{\R}{\mathbb{R}}
\newcommand{\Z}{\mathbb{Z}}

\newtheorem{theorem}{Theorem}[section] 
\newtheorem{lemma}[theorem]{Lemma}     
\newtheorem{corollary}[theorem]{Corollary}
\newtheorem{proposition}[theorem]{Proposition}

\newtheorem{assumption}[theorem]{Assumption}

\theoremstyle{definition}
\newtheorem{definition}[theorem]{Definition}
\newtheorem{example}[theorem]{Example}

\newtheorem*{ack*}{Acknowledgment}

\theoremstyle{remark}
\newtheorem{remark}[theorem]{Remark}

\numberwithin{equation}{section}

\title[Asymptotic pressure on some self-similar trees]{{Asymptotic pressure on some self-similar trees}}

\author{Karl Petersen}
\address{Department of Mathematics,
	CB 3250 Phillips Hall,
	University of North Carolina,
	Chapel Hill, NC 27599 USA}
\email{petersen@math.unc.edu}

\author{Ibrahim Salama}
\address{School of Business, North Carolina Central University, Durham, NC 27707 USA}
\email{isalama@nccu.edu}

\date{\today}

\begin{document}
	
	
	\begin{abstract}
		The vertices of the Cayley graph of a finitely generated semigroup form a set of sites which can be labeled by elements of a finite alphabet in a manner governed by a nonnegative real {\em interaction matrix}, respecting nearest neighbor adjacency restrictions. 
				To the set of these configurations one can associate a {\em pressure}, which is defined as the limit, when it exists, of averages of the logarithm of the partition function over certain finite subgraphs. 
		We prove that for shifts of finite type on generalized Fibonacci trees {and many primitive interaction matrices}, the limit exists and is given by an infinite series.
		We also show that the limit of any cluster points of the pressure on finite subtrees as the number of generators grows without bound, which we call the {\em asymptotic pressure}, equals the logarithm of the maximum row sum of the interaction matrix. 
	\end{abstract}
	\maketitle  
	
	\section{Introduction}\label{sec:intro}
	Entropy and its generalization pressure are basic concepts imported from statistical physics into information theory and dynamical systems. 
	Usually the state space has been the set of labelings by elements of a finite alphabet of the vertices of the integer lattice $\Z^d$, which then acts on this set of configurations by shifts in the coordinate directions. 
	There has also been interest in such thermodynamical models related to labelings of other graphs, including trees and Cayley graphs of other groups and semigroups, and even arbitrary countable sets, see for example 
	 \cites{BergerYe1990,BergerYe1998,BP1994,BPS1995,BW2002,BW2004,Eggarter1974,
	 	MairesseMarcovici2017,MHZ1974,Preston1974,Runnels1967,Spitzer1975}.
 	 {Burton, Pfister, and Steif \cite{BPS1995}
 		gave an infinite series formula for the pressure on the free group or semigroup on two generators and showed that
 		the variational principle for pressure can fail on trees; 
 		in fact there is an equilibrium state, which is a Gibbs state, if and only if the interaction matrix has constant row sums.} 
 	(Although they treated the case of the free group on two generators and the full $2$-tree, they mentioned that the results would extend to all homogeneous $k$-trees with $k \geq 3$.)
	Piantadosi \cites{PiantadosiThesis,Piantadosi2008} and the current authors \cites{PS2018,PS2019},
	  unaware of the 
	 work of Burton, Pfister, and Steif \cite{BPS1995},
	obtained infinite series formulas for the topological entropy of the golden mean (also called ``hard core" or ``hard square") shift of finite type (SFT) on the free group and free semigroup on a finite number $k$ of generators.
	 The current authors \cites{PS2018,PS2019} also proved the existence of topological entropy for {subshifts} on trees and, for the golden mean subshift, convergence to it by strip approximations analogous to those of Marcus and Pavlov  \cites{Pavlov2012,MarcusPavlov2013}.

	 	 	Louidor, Marcus, and Pavlov \cite{LouidorMarcusPavlov2013} and 
	 	 Meyerovitch and Pavlov \cite{MeyerovitchPavlov2014} studied what they called {\em limiting entropy} and {\em independence entropy}, proving that they coincide for $\Z$ subshifts.
	Based on numerical evidence, Piantadosi conjectured \cite{PiantadosiThesis}*{Conjecture 4.1} that  
	the topological entropy $h^{(k)}$ of the golden mean SFT on the free group on $k$ generators has limiting value $\log 2$ as $k \to \infty$ (see also \cite{Piantadosi2008}). 
	The current authors \cite{PS2018} (see also \cite{PS2019}) proved that the topological entropy of a tree shift on the regular $k$-tree defined by an irreducible $d \times d$ $0,1$ matrix with maximum row sum $s$ has limit $\log s$ as $k \to \infty$, {giving a positive answer to Piantadosi's conjecture.}
	 We will call this limit the {\em asymptotic entropy} 
	rather than ``limiting entropy", 
	because {when they exist} entropy and pressure for such systems are already limits or cluster points of functions of sums over finite configurations as the size grows without bound.
	
Here we extend Piantadosi's observations and our previous result in three ways:
the set of sites is allowed to be a tree with branching restrictions, as considered by Ban et al. \cites{BC4,BC5,BC6,BC8,BC7}, equivalently the Cayley graph of a finitely generated semigroup with certain relations;
the labeling is governed by general nearest neighbor shift of finite type (SFT) restrictions, beyond the golden mean model;
and beyond entropy we consider asymptotic {\em pressure}.
In Section \ref{sec:existence} we follow the plan of \cite{BPS1995} to prove the existence of pressure 
\be
P^{(k)}=\lim_{n \to \infty} P_n^{(k)}
\en
 on any generalized Fibonacci $k$-tree {for many primitive interaction matrices} and give an infinite series formula for it (Theorem \ref{th:pressureexists}). 
 
In Section \ref{sec:aspres} we show that for a fixed interaction matrix, a fairly general sequence of restricted trees on $k$ generators, and any cluster points $Q^{(k)}$ in $n$ of $P_n^{(k)}$, the {\em asymptotic pressure} $\lim_{k \to \infty} Q^{(k)}$ of the sequence is the logarithm of the maximum row sum of the interaction matrix.	
{The reason for this seems to be the increasing relative importance of the contribution to the pressure by edges at the boundary of a subtree as the valence (dimension) grows,
and that it is possible to find many valid configurations which label the next-to-last row of the subtree  with a symbol whose corresponding row in the interaction matrix achieves the maximal row sum (see Lemma \ref{lem:sizes} and Theorem \ref{thm:limit}). }
	
	\section{Setup}\label{sec:setup}
	
	We consider entropy and pressure for sets of configurations on a set of sites. 
	There will be two adjacency matrices involved: a primitive (or more generally irreducible) $0,1$ matrix $R$ that is used to {\em form} the tree by restricting the ways that it can branch, and a nonnegative $d \times d$ matrix $A$ that will restrict the ways that the tree can be {\em labeled} by elements of an alphabet of $d$ symbols.. 
	In statistical physics often the set of sites is the $k$-dimensional integer lattice $\Z ^k$. 
	In previous work \cites{PS2018,PS2019} we have studied entropy on the rooted $k$-tree, which is the Cayley graph of the free semigroup on $k$ generators $g_1,\dots ,g_k$. 
	Each string $u$ on the alphabet $\{g_1, \dots, g_k\}$ corresponds to an element of the semigroup and a vertex of the graph, with the empty string $\epsilon$ corresponding to the root. 
	For each string $u$ and $i=1,\dots,k$ there is a directed edge from $u$ to $u g_i$.
	
	Here we consider a general setting, considered previously by Ban et al.   \cites{BC4,BC5,BC6,BC8,BC7} of configurations on the subtree determined by a $k \times k$ $0,1$ {\em restriction matrix} $R$:
		\be\label{eq:subtree}
	S=	S(k,R)=\{0\} \cup \{g_{i_1} \dots g_{i_n}: n \geq 1, \text{ no } R_{{i_j},{i_{j+1}}}=0, j=1, \dots,
			n-1\}.
			\en
		The subtree $S$ can be regarded as the Cayley graph of the semigroup with right absorbing element $0$ 
		($g0=0$ for all $g \in S$) 
		 generated by $\{g_1,\dots ,g_k\}$ and with relations $g_i g_j =0$ if $R_{i j}=0$. 
		Vertices of the graph are identified with elements of the semigroup and with strings on the generators {that do not reduce to $0$}, and the
		 absorbing element corresponds to the root of the tree.
		Alternatively, $S$ can be regarded as the semigroup generated by $\{ g_1, \dots , g_k\}	$ with relations $g_i g_j =g_i$ if $R_{i j}=0$ and identity element $\epsilon$ at the root, cf. \cites{BC4,BC5,BC6,BC8,BC7}.

	We assume at first that each restriction matrix $R$ is {\em primitive}, meaning that some positive power has all entries positive.
	\begin{example}\label{ex:restriction matrices}
  For $0 \leq r<k$ the matrix $R(k,r)$ defined by $R(k,r)_{ij}=0$ if and only if each of $i,j > k-r$, otherwise $R(k,r)_{ij}=1$, defines a ``generalized Fibonacci tree".
  (If $r=0$, so that $R_{ij}=1$ for all $i,j$, $S$ is the free semigroup on $k$ generators and its Cayley graph is the full $k$-tree.)
\end{example}

	We define the {\em height} $|g|$ of an element $g \in S$ to be the length of the {shortest} word that represents it (so that the root has height $0$). 
	We will also use $|.|$ to denote the cardinality of a set as well as the sum of the absolute values of all entries of a matrix or vector.
	
		For each $n \geq 0$, for readability suppressing $k$, we define 
		\be
		\Delta_n=\{g \in S: |g| \leq n\} \quad\text{ and } \quad L_n=\{g \in S: |g|=n\}.
		\en
	Then $|L_1|=k$ and, for $n \geq 2$, 
	\be
	|L_n|= |R^{n-1}|=\sum_{i,j} (R^{n-1})_{ij}.
		\en
		(This holds also for $n=1$, with $R^0=$ the $k \times k$ identity matrix.)
		Further,
		\be
		|\Delta_n|=1+\sum_{j=1}^n |R^{j-1}|.
		\en
	
	A {\em configuration} on $S$ is a labeling of 
the sites (vertices of the subtree) by elements of a finite alphabet $D=\{ 1,2, \dots, d\}$, thus an element $x \in X=D^S$. 
  We will be interested in labelings that are allowed by {\em nearest neighbor} constraints.

	Let $A=(a_{ij})$ be a nonnegative real $d \times d$ matrix that we take as specifying {\em pair interactions} and let $w=(w_j)$ be a positive real vector that we take as specifying {\em site energies}. 
	By discarding irrelevant states we may assume that $A$ is {\em nondegenerate} in the sense that no row nor column is identically zero. 
	The idea is that if a vertex $h$ is assigned label $j$ by a configuration $x \in X$, then the ``particle" $j$ at site $h$ is given an {energy $\log w(j)$ by some ambient field}.
	And if vertices $g$, with label $i$, and $h$, {with label $j$}, are joined by an {\em edge} (so that $h=gg_s$ for some $s=1,\dots ,k$), then they experience an interaction (tension, attraction, or repulsion) determined by $a_{ij}$.
	The {\em interaction matrix}
	\be
	E(i,j)=a_{ij}w_j
	\en
	collects these effects. 
	(For improved readability we will sometimes denote matrix or vector indices parenthetically rather than by subscripts.)

	In relation to the setup in \cite{BPS1995}, 
	\be
	a_{ij}=e^{\phi(i,j)}, \quad w_j=e^{\chi(j)};
	\en
	but note that we allow $\phi(i,j)=-\infty$, while $\chi(j) \in \R$. 

	 Now we restrict our attention to labelings that conform to
	the  adjacency restrictions provided by the matrix $A$.
	Thus our set of configurations will be the {(hom---see \cite{CM2018}) {\em tree shift of finite type} $X_A$ determined by $R$ and $A$ contained in the product space
	$X=D^S$:
	\be\label{eq:transitions}
	\begin{gathered}
	X_A=\{x \in D^S: A(x(g),x(h))>0 \text{ for all } g,h\in S \text{ such that }\\
	 h=gg_s \text{ for some } s=1,\dots ,k\}.
	 \end{gathered}
	\en
	(We ignore every configuration that has an interaction of size $0$ between some pair of adjacent sites.) 
	{Because we have assumed that $A$ is nondegenerate, $X_A$ is nonempty.}
	
	$X$ (as well as $X_A$) is a compact metric space with distance $d(x,y)=1/(n+1)$ if $n=\max\{j:x=y \text{ on } \Delta_j\}$.
	The semigroup $S$ acts continuously on $X$ according to
	\be
	(x g)(h)=x(gh).
	\en
	
{We are interested in (allegedly) physical quantities due to configurations $x \in D^{\Delta_n}$ on finite subtrees $\Delta_n$ that are the restrictions of configurations in $X_A$, and then their limits as $n \to \infty$.
		Define $X_A^{(n)}$ to be the set of labelings of $\Delta_n$ that are restrictions to $\Delta_n$ of some element of $X_A$. 
	The contribution to the pressure from configurations on $\Delta_n$ that have symbol $i \in D$ at the root is 
	\be
	Z_n(i)=w(i) \sum_{\substack{x \in X_A^{(n)}\\x(\epsilon)=i}} \prod_{\substack{\gamma=\left< g,h\right> \\ \text{ edge in } \Delta_n}} E(x(g),x(h)).
	\en
	We continue to suppress $k$ much of the time when it is fixed to avoid excess notation.
	The $n$'th {\em partition vector} is
	\be
	Z_n=(Z_n(1),\dots,Z_n(d)), 
	\en
	the {\em partition function on $\Delta_n$} is
	\be
	|Z_n|=Z_n(1)+ \dots + Z_n(d),
	\en
	and the $n$'th {\em pressure probability vector} is $\rho_n$ defined by
	\be
	\rho_n(i)=\frac{Z_n(i)}{|Z_n|}, \quad i=1, \dots, d.
	\en
	The {\em pressure} or {\em free energy} on $\Delta_n$ is
	\be
	P_n^{(k)}=\frac{\log |Z_n|}{|\Delta_n|}.
	\en
		The {\em upper limiting pressure} $\bar P^{(k)}$ and {\em lower limiting pressure} $\underline P^{(k)}$ are defined by
		\be\label{eq:pressures}
		\bar P^{(k)}=\limsup_{n \to \infty} \frac{\log |Z_n|}{|\Delta_n|}, \quad \underline P^{(k)}=\liminf_{n \to \infty} \frac{\log |Z_n|}{|\Delta_n|}.
				\en
			}
					
		{When $A$ is a $0,1$ matrix and $w_j=1$ for all $j$, 
		$|Z_n|$ is the number of configurations on $\Delta_n$ that are allowed by the adjacency matrix $A$ (symbols $i,j \in D$ are allowed at vertices connected by an edge if and only if $a_{ij}=1$). 
			In this case $\lim_{n \to \infty}P_n^{(k)}$, if it exists, is the {\em topological entropy} of the tree shift of finite type determined by $A$ on the restricted tree.
		The limit has been proved to exist for any tree shift (not necessarily finite type) on a full $k$-tree \cites{PS2018,PS2019}. 
		Ban et al. \cites{BC4,BC5,BC6,BC8,BC7} proved existence of the topological entropy for shifts of finite type on classes of restricted trees, including the generalized Fibonacci trees of Example \ref{ex:restriction matrices}. 
	In the following section we extend these results to pressure on such trees.}

\section{Existence of pressure on generalized Fibonacci trees and an infinite series formula for it}\label{sec:existence}
	{{
		To formulate a recursion formula for the partition function in somewhat compressed notation,} we use unusual (Hadamard or Schur) coordinatewise products and powers for (usually column) vectors $v=(v_i)$ and matrices $A=(a_{ij})$: 
	\be
	(v \times w)_i=v_iw_i, \quad (A^{[n]})_{ij}=a_{ij}^n.
	\en
	In this section $A$ is a nonnegative $d \times d$ primitive matrix.	
}

	{Consider first the case of the full $k$-tree ($r=0$), with symbol $i \in D$ at the root. 
	If $n \geq 2$, then each node of the first row $L_1$ of $\Delta_1$ can be thought of as the root of a subtree $\Delta_{n-1}$ of height $n-1$ and can be assigned any symbol from $D$ that can follow $i$ according to the transitions allowed by $A$. 
	The labelings of these subtrees are independent of one another, so one may interchange sum and product as in the following.
		\be
		\begin{aligned}
	Z_n(i)&=w(i) \sum_{\substack{x \in X_A^{(n)}\\x(\epsilon)=i}} \prod_{\substack{\gamma=\left< g,h\right> \\ \text{ edge in } \Delta_n}} E(x(g),x(h))\\
	&=	
	w(i) \sum_{\substack{x \in X_A^{(n)}\\x(\epsilon)=i}} \prod_{\substack{\gamma=\left< g,h\right> \\ \text{ edge in } \Delta_n}} A(x(g),x(h)) w(x(h))\\
	&=w(i) \left[ \sum_{j=1}^d A(i,j) \left(w(j)  \sum_{\substack{x \in X_A^{(n-1)}\\x(\epsilon)=j}} \prod_{\substack{\gamma=\left< g,h\right> \\ \text{ edge in } \Delta_{n-1}}} E(x(g),x(h))\right) \right] ^k \\
	&=w(i) \left[ \sum_{j=1}^d A(i,j) Z_{n-1}(j)\right] ^k
	=\left( w \times \left[AZ_{n-1}\right] ^{[k]}\right)(i).
	\end{aligned}
	\en
}
	
{Now consider a {\em generalized Fibonacci} tree defined by a restriction matrix $R=R(k,r), k \geq 2, 0 \leq r<k$, with $R_{ij}=0$ if each of $i,j>k-r$, otherwise $R_{ij}=1$. 
	On row $L_1$ of $\Delta_n$ there are now $k-r$ ``free" vertices, each the root of a subtree $\Delta_{n-1}$ of height $n-1$, along with $r$ ``restricted" vertices, each followed by $k-r$ edges whose terminal vertices  are then the roots of subtrees $\Delta_{n-2}$ of height $n-2$. 
	These subtrees are labeled independently. 
	Letting
	\be
	Z_0=w, \quad Z_1=w \times (AZ_0)^{[k]},
	\en
	and looking at $\Delta_n$ and its first few rows, we thus find that for $n \geq 1$,}
	\be
	\begin{aligned}
		Z_{n}(i)&=w(i) \left( \sum_jA(i,j)Z_{n-1}(j)\right)^{[k-r]} \left[\sum_m A(i,m)w(m)\left(\sum_u A(m,u)Z_{n-2}(u)\right)^{[k-r]}\right]^{[r]}\\
		&=w(i) \left(\left( AZ_{n-1}\right)(i)\right)^{[k-r]} \left[ \left(A \left( w \times (AZ_{n-2})^{[k-r]}\right)\right)(i)\right]^{[r]},
	\end{aligned}
	\en
	so that 
	\be\label{eq:iter}
	\begin{aligned}
		|Z_{n}|&= \left| w \times (AZ_{n-1})^{[k-r]} \times \left[ A  \left( w \times (AZ_{n-2})^{[k-r]}\right)\right]^{[r]}\right|
		&=  |Z_{n-1}|^{[k-r]} |Z_{n-2}|^{[r(k-r)]} g(n),
	\end{aligned}
	\en
	with
	\be\label{eq:gspecial}
	{g(n)=\left|w \times (A\rho_{n-1})^{[k-r]} \times \left[ A  \left( w \times (A\rho_{n-2})^{[k-r]}\right)\right]^{[r]}\right|.}
	\en
	
	{More generally, for $d$-dimensional probability vectors $\alpha,\beta$ define
	\be
	g(\alpha,\beta)=\left|w \times (A\alpha)^{[k-r]} \times \left[ A  \left( w \times (A\beta)^{[k-r]}\right)\right]^{[r]}\right|,
	\en
	and on the set of pairs of such vectors for which $g(\alpha,\beta)>0$ define
	\be
	T(\alpha,\beta)=\left( w \times (A\alpha)^{[k-r]} \times \left[ A  \left( w \times (A\beta)^{[k-r]}\right)\right]^{[r]} /g(\alpha,\beta),\alpha\right).
	\en
	\begin{remark}\label{rem:g}
		{The function $g(n)$ is bounded above; but it might take values less than $1$, and this could cause problems for convergence of the infinite series for pressure (\ref{eq:series1}).}
		\begin{enumerate}
			\item			In the case of a full $k$-tree ($r=0$), $\log g$ is bounded because  
				$w>0$ and every column of $A$ has a positive entry. 
				(To see this, note that $\rho_{n-1}$ is a probability vector with positive entries, at least one of which, say with index $j$, is at least $1/d$. 
				We may choose $i$ such that $a_{ij}\geq a=\min\{a_{ij}:a_{ij}>0\}>0$.
				Then $(A \rho_{n-1})(i)\geq a/d$ and $g(n)=|w \times (A \rho_{n-1})^{[k]}| \geq (\inf_i w_i) (a/d)^k$.)
				\item For any generalized Fibonacci tree, if $A$ has a positive row then $\log g$ is bounded. 
				(The statement is clear if for some $i$ all $a_{ij}=1$, hence also if all $a_{ij}=\delta$ for some $\delta>0$, hence also if all $a_{ij}\geq \delta$ for some $\delta>0$.) 
				\item If $T^n (\rho_1,\rho_0)$ converges to a fixed point or periodic orbit, then $\log g$ is bounded. 
				\item
				If $g(\alpha,\beta)>0$ on the closure of the orbit of $(\rho_1,\rho_0)$ under $T$, then it is bounded below by a positive constant and hence again $\log g$ is bounded.
				\end{enumerate}
				\end{remark} }
								{We show now that if $A$ satisfies a condition apparently stronger than primitive, then $\log g$ is bounded.}
				
	Suppose for the moment that we are dealing with a general restricted tree determined by a $k \times k$ restriction matrix $R$, as in (\ref{eq:subtree}), and not just a generalized Fibonacci tree. 
	For a fixed $p>0$, the subtrees rooted at the vertices $v \in L_p$ have different types, depending on the directions $e(v) \in \{1, \dots, k\}$ of their incoming edges to their root vertices $v$: only initial edges in directions $h$ for which $R(e(v),h)=1$ are allowed.
	For $n>p$ denote by 
	$X_A^{(n-p,v)}$ the set of restrictions of configurations in $X_A^{(n)}$ to the restricted subtree $\Delta_{n-p}^{(v)}$ of $\Delta_n$ that has root at $v$, and by	
	\be
		Z_{n-p}^{(v)}(j)= \sum_{\substack{x \in X_A^{(n-p,v)}\\x(\epsilon)=j}} \prod_{\substack{\gamma=\left< g,h\right> \\ \text{ edge in } \Delta_{n-p}^{(v)}}} E(x(g),x(h))
		\en
		the pressure on the subtree $\Delta_{n-p}^{(v)}$ of $\Delta_n$ that has label $j$ at its root, $v$. 
		
		Now consider the set $L_p(i)$ of strings of symbols $J=(j(v), v \in L_p)$ that occur on $L_p$ in configurations $x \in X_A^{(n)}$ on $\Delta_n$ that have symbol $i$ at the root: $x(\epsilon)=i, x(v)=j(v)$ for $v \in L_p$. 
		We are interested in matrices $A$ that guarantee that for some $p$ we have $L_p(i)=L_p(j)$ for all $i,j \in \{1, \dots, d\}$.
		\begin{definition}\label{def:primitive}
			Let $d,k \geq 1$ and let $A$ be a nondegenerate irreducible $d \times d$ nonnegative matrix that determines allowed transitions on trees, as in Equation \ref{eq:transitions}. 
			If there is an $n \geq 1$ such that on a tree $t$ with labelings allowed by $A$ we have $L_n(i)=L_n(j)$ for all $i,j \in D=\{1, \dots, d\}$, we will say that $A$ is {\em $n$-primitive on $t$}. 
			We denote the set of $n$-primitive $d \times d$ matrices on the full $k$-tree by $P(k,n)$.\\
			\indent
			{If there is an $n \geq 1$ such that for all $i \in D$ we have $L_n(i)=D^{|L_n|}$, we will say that $A$ is \em{$n^*$-primitive on $t$}.
					We denote the set of $n^*$-primitive matrices on the full $k$-tree by $P^*(k,n)$.}
			\end{definition}
	{	\begin{remark}\label{rem:primitive}
			{The following observations are included to provide some familiarity with the condition that $A \in P(k,n)$.}
			\begin{enumerate}
				\item {For all $k,n$ we have $P^*(k,n) \subset P(k,n)$.}
				\item If $A \in P(k,n)$ for some $k \geq 1$, then $A$ is primitive, and in fact $A^n>0$.
				\item If $A \in P(k,n)$, then $A$ is $n$-primitive on every subtree of the full $k$-tree.
				\item $P(k+1,n) \subset P(k,n)$.
				\item $P(k,n) \subset P(k,n+1)$.
				\item The $3 \times 3$ matrix $A_2=(110,101,101)$ is in $P(2,2) \setminus P(2,1)$.
				\item $A_3=(110,101,100) \in P(2,3) \setminus P(2,2)$.
				\item If $A^n>0$ and $A$ has a positive row, then $A\in P(k,n+1)$.
				\end{enumerate}
				\end{remark} }
				
		We return now to a general restricted tree and strings $J=(j(v), v \in L_p) \in L_p(i)$ on row $p$ of $\Delta_n$ in the tree. 
		For such a string $J$ define
		\be
	{F_n(i,J)} = \prod_{v} Z_{n-p}^{(v)} (j(v)), 
		\en
		{to be the sum over all configurations on $\cup_{v \in L_p} \Delta_{n-p}$ that have $J$ on $L_p$ of the products of the edge weights (the product of sums is a sum of products---see (\ref{eq:switch}) below)},
		and let
		\be
	{M_n(i)=\max_{\text{allowable } J} F_n(i,J).}
		\en
		Choose one allowable string $J_0=(j_0(v), v \in L_p)$ that achieves the maximum value $M_n$ of {$F_n(i,J)$ over all $i$ and $J$.}
		
	\begin{remark}\label{rem:same}	{If $A \in P(k,p)$ for some $p$, then 
 			{for all choices of symbol $i$ at the root and all $n \geq p$,
				we have $J_0 \in L_p(i)$ and $M_n(i)=M_n$ .}}
				\end{remark}
	
	\begin{proposition}
		Suppose that $p \geq 1$ and $A \in P(k,p)$, so that $A^p>0$. 
		For a general restricted tree, 
				there are {positive} constants $\eta_p$ and $\xi_p$ such that for all $n>p$ there is $M_n>0$ such that for each $i=1,\dots,d$,
		\be\label{eq:ineq}
		\xi_p M_n \leq Z_n(i) \leq \eta_p M_n.
		\en
	\end{proposition}
	\begin{proof}
						{We form some special configurations on $\Delta_n$.
				Given $i \in \{1, \dots, d\}$,
				{because $J_0 \in L_p(i)$ we can }
				choose a labeling $x \in X_A^{(p)}$ that has $i$ at the root and assigns labels $j_0(v)$ to the vertices $v \in L_p$ 
				{(see Remark \ref{rem:same})}. 
				Let 
		\be
		\xi_p(i,x)= w(i) \prod_{\substack{\gamma=\left< g,h\right> \\ \text{ edge in } \Delta_p}} E(x(g),x(h))
		\en
		denote the product of $w(i)$ and the weights determined by $x$ on the edges in $\Delta_p$,
		 and let
		 \be
		\xi_p= \inf \{\xi_p(i,x): x \in X_A^{(p)}, x(\epsilon)=i, i=1,\dots,d\}.
	\en 
Now fix $i \in \{ 1, \dots, d\}$ and choose any $x \in X_A^{(p)}$ with $x(\epsilon)=i$ and $x(v)=j_0(v)$ for all $v \in L_p$.}
	{We label $\Delta_p$ by $x$, and
	then for each $\Delta_{n-p}$ with a vertex on $L_p$ we assign label $j_0(v)$ to its root $v$ and any labeling allowed by $A$ to the rest of its vertices. These labelings of the $\Delta_{n-p}$ are assigned independently of each other.}
	
		{The sum over {\em all} configurations on $\Delta_n$ with $i$ at the root of the products of the weights on the edges in $\Delta_n$ is greater than or equal to the sum over just these special configurations
			(that have $x$ on $\Delta_p$ and $J_0$ on $L_p$),
			so {(recall Remark \ref{rem:same})}
				\be
		Z_n(i) \geq \xi_p(i) {F_n(i,J_0)} \geq \xi_p M_n,
		\en 
	proving the left-hand inequality.}
		
			For the right-hand inequality, 
			we decompose each configuration $x\in X_A^{(n)}$ into a configuration $x_1$ on $\Delta_p$ and a configuration $x_2$ on $\cup_{v \in L_p} \Delta_{n-p}^{(v)}$. 
			These configurations are not independent, since they overlap in $L_p$; this causes the first inequality below.			
			Denote by $\eta_p$ {the sum}, over all allowed configurations $x$ on $\Delta_p$, of the {product of $w$ at the root and the weights} determined by $x$ on the edges in $\Delta_p$,
			multiplied by {$d$}:
			\be
			{\eta_p = d \sum_{x \in X_A^{(p)}} {w(x(\epsilon))}\prod_{\substack{\gamma=\left< g,h\right> \\ \text{ edge in } \Delta_p}} E(x(g),x(h)).}
			\en
					Then
					\be
					\begin{aligned}
    Z_n(i)&=w(i)\sum_{\substack{x\in X_A^{(n)}\\{x(\epsilon)=i}}} \prod_{\substack{\gamma=\left< g,h\right> \\ \text{ edge in } \Delta_n}} E(x(g),x(h)) \\
    &\leq w(i)  \sum_{\substack{x_1\in X_A^{{(p)}}\\{x_1(\epsilon)=i}}} \sum_{\substack{x_2 \text{ on }\\ \cup_{v \in L_p} \Delta_{n-p}^{(v)}}} \prod_{\substack{\gamma=\left< g,h\right> \\ \text{ edge in } \Delta_n}} E(x(g),x(h)) \\
  &=w(i) \sum_{\substack{x_1\in X_A^{({p})}\\{x_1(\epsilon)=i}}} \prod_{\substack{\gamma=\left< g,h\right> \\ \text{ edge in } \Delta_p}} E(x(g),x(h)) 
    \sum_{\substack{x_2 \text{ on }\\ \cup_{v \in L_p} \Delta_{n-p}^{(v)}}} \prod_{v \in L_p} \prod_{\substack{\gamma=\left< g,h\right> \\ \text{ edge in } \Delta_{n-p}^{(v)}}} E(x(g),x(h))\\
    &\leq \frac{\eta_p}{d}  \sum_{\substack{x_2 \text{ on }\\ \cup_{v \in L_p} \Delta_{n-p}^{(v)}}}  \prod_{v \in L_p} \prod_{\substack{\gamma=\left< g,h\right> \\ \text{ edge in } \Delta_{n-p}^{(v)}}} E(x(g),x(h))
   =\frac{\eta_p}{d}  \prod_{v \in L_p} \sum_{\substack{x_3 \text{ on }\\  \Delta_{n-p}^{(v)} }}    \prod_{\substack{\gamma=\left< g,h\right> \\ \text{ edge in } \Delta_{n-p}^{(v)}}} E(x(g),x(h))\\
    &= \frac{\eta_p}{d}  \prod_{v \in L_p} \sum_{j=1}^d \sum_{\substack{x_3 \text{ on }\\  \Delta_{n-p}^{(v)} \\ x_3(v)=j}}    \prod_{\substack{\gamma=\left< g,h\right> \\ \text{ edge in } \Delta_{n-p}^{(v)}}} E(x(g),x(h))
    = \frac{\eta_p}{d}  \prod_{v \in L_p} \sum_{j=1}^d   Z_{n-p}^{(v)}(j)\\
   &= {\frac{\eta_p}{d}   \sum_{j=1}^d  \prod_{v \in L_p}  Z_{n-p}^{(v)}(j)
    \leq \frac{\eta_p}{d} d \prod_{v \in L_p} Z_{n-p}^{(v)}(j_0(v))
     = \eta_p M_n.}
    \end{aligned}
\en
\end{proof}
The interchange of sum and product in the fourth line of the above calculation 
{(and elsewhere)} takes some thought, but it seems to be correct, 
{since the configurations on each $\Delta_{n-p}^{(v)}$ are assigned independently}. 
The figure with two small trees and nodes labeled by $a,b,c,d,e$ taking values in $\{1, \dots, d\}$ is meant to give the idea: by taking out a common factor we have
\be\label{eq:switch}
\begin{gathered}
\sum_{a,b,c,d,e} E(a,b)E(a,c)E(d,e)=\sum_{a,b,c}\sum_{d,e}E(a,b)E(a,c)E(d,e)\\
=\left(\sum_{d,e} E(d,e)\right) \left(\sum_{a,b,c}E(a,b)E(a,c)\right).
\end{gathered}
\en
		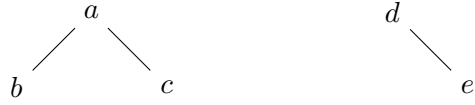
\begin{figure} \label{fig:trees}
		\begin{tikzpicture}[scale=1]
			\node at (-2,0) (a) {$a$};
			\node at (-3,-1) (b) {$b$};
			\node at (-1,-1) (c) {$c$};
			\draw (a) -- (b);
			\draw (a) -- (c);
			
			\node at (2,0) (d) {$d$};
			\node at (3,-1) (e) {$e$};
			\draw (d) -- (e);
		
			\end{tikzpicture}
		\caption{Two small trees}
		\label{fig:ig}
	\end{figure}
	
	\begin{corollary}\label{cor:rat}
		Suppose that $p \geq 1$ and $A \in P(k,p)$.
		Then for a general restricted tree, 
		for each $i,j=1,\dots,d$ and $n>p$,
		\be
		\frac{Z_n(i)}{Z_n(j)}\geq \frac{\xi_p}{\eta_p}.
		\en
	\end{corollary}

For convenience we write $k_1=k-r$ and $k_2=r$. 
	
	\begin{corollary}\label{cor:g}
		For a generalized Fibonacci tree with $A \in P(k,p)$ and 
		\be
		{G_m}=w \times (A \rho_{m-1})^{[k_1]} \times [A\left(w \times (A \rho_{m-2})^{[k_1]}\right)]^{[k_2]},
		\en
		there is a constant $c>0$ such that 
		\be
		g(m)=|G_m| \geq c>0 \quad\text{ for all } m>p+2.
		\en
	\end{corollary}
	\begin{proof}
		By Equation \ref{eq:ineq}, each $\rho_m(i) \geq 1/d|Z_p|$. Let $a_i=\inf_i\sum_ja_{ij}$ and $b=\inf_iw_i$. Then for each $i$,
		\be
		G_m(i) \geq b \left( \frac{a_1 \xi_p}{d\, \eta_p} \right) ^{k_1} \left[ a_1 \left(b \left(\frac{a_1 \xi_p}{d\, \eta_p}\right)\right)^{k_1}\right]^{k_2} >0, 
		\en
		independently of $m$.
	\end{proof}
	
	\begin{remark}
		\begin{enumerate}
			\item The conclusion of Corollary \ref{cor:g} holds also for some matrices $A$ that are not in $P(k,n)$ for any $n$. 
			For example, if $d=3$, {$w(i)=1$ for all $i=1,\dots,d$, and $A_5=(011,101,110)$, then, because $A_5$ has constant row sums, $\rho_n(i)=1/d$ for all $n$ and $i$, and hence $g(n)$ is the same positive constant for all $n$.} 
			{For general positive $w$, we may bound $g(n)$ from below by replacing $w$ by the constant vector $\tilde w(i) = \inf_j w(j)>0$ for all $i$.} 
		\item Maybe the conclusion holds for all primitive matrices $A$.
		 \item Numerical calculations indicate that if $A$ is assumed only to be irreducible, there might not be a positive lower bound for $\rho_n(i)$, yet there is one for $g(n)$.
		 \item 	We do not see how to deduce directly from Equation \ref{eq:ineq} that $g(n) = |Z_n|/(|Z_{n-1}|^{[k-r]}|Z_{n-2}|^{[r(k-r)]})$ is bounded below by a positive constant for all general restricted trees, without using the particular form of the recurrences (\ref{eq:iter}) and (\ref{eq:gspecial}) for $Z_n$ and $\rho_n$ for generalized Fibonacci trees.
			\end{enumerate}
	\end{remark}
 
 We proceed now to establish an infinite series formula for pressure on generalized Fibonacci trees, which also proves existence of the limit that defines it.

	In order to iterate \ref{eq:iter}, let us abbreviate
	\be
	|Z_0|=a,|Z_1|=b,u=k-r,v=ru.
	\en
	Thus 
	\be\begin{aligned}
		|Z_2|&=a^vb^ug(1)\\
		|Z_3|&=a^{uv}b^{u^2+v} g(1)^u g(2)\\
		|Z_4|&= a^{u^2v+v^2}b^{u^3+2uv} g(1)^{u^2+v} g(2)^u g(3)\\
		|Z_5|&=a^{u^3v+2uv^2}b^{u^4+3u^2v+v^2} g(1) ^{u^3+2uv} g(2)^{u^2+v} g(3)^u g(4) \\
		|Z_6|&=a^{u^4v+3u^2v^2+v^3}b^{u^5+4u^3v+3uv^2}g(1)^{u^4+3u^2v+v^2} g(2)^{u^3+2uv} g(3)^{u^2+v} g(4)^u g(5)\\
		&\text{etc.}
	\end{aligned}
	\en
	Let us define $c(n),d(n),e(n)$ for $n \geq 0$ by
		\be\label{eq:Z_n}
	|Z_n|=a^{c(n)} b^{d(n)} g(1)^{e(n-1)} \dots g(n-2)^{e(2)}g(n-1)^{e(1)},
	\en
	starting with $c(1)=0,d(1)=1,e(0)=0,e(1)=1$. Then
	\be\label{eq:cderec}
	\begin{bmatrix}
		c(n)\\
		d(n)
	\end{bmatrix}
	= \begin{bmatrix} 0 & v \\
		1 & u \end{bmatrix}^{n-1} 
	\begin{bmatrix} 0\\ 
		1
	\end{bmatrix}, \quad \text{ and } \begin{bmatrix}
		e(n)\\
		e(n-1)
	\end{bmatrix} 
	= \begin{bmatrix}
		u & v\\
		1 & 0
	\end{bmatrix}^{n-1}
	\begin{bmatrix} 
		1 \\
		0
	\end{bmatrix}.
	\en
	
	{The following Proposition shows that $c(n), d(n), e(n)$ and $|L_n|$ all have the same exponential growth rate.}
	\begin{proposition}
	{	The matrices $R(k,r)$ and 
		\be
		\begin{gathered}
		M_1=\begin{bmatrix} 0 & v \\
			1 & u \end{bmatrix} = \begin{bmatrix} 0 & r(k-r)\\1 & k-r \end{bmatrix},
		M_2=\begin{bmatrix}
			u & v\\
			1 & 0
		\end{bmatrix} = \begin{bmatrix} k-r & r(k-r)\\1 & 0 \end{bmatrix}, \text{ and }\\
	M_3=\begin{bmatrix} u & u \\
		r & 0\end{bmatrix} = \begin{bmatrix} k-r & k-r\\r & 0 \end{bmatrix}
	\end{gathered}
		\en
	 have the same maximal eigenvalue.}
	\end{proposition}
	\begin{proof}
		$M_1,M_2,M_3$ all have the same characteristic polynomial $-x(u-x)-v$ and
	eigenvalues 
		\be
		\beta=\frac{1}{2}(u+\sqrt {u^2+4v}), \quad \alpha=\frac{1}{2}(u-\sqrt {u^2+4v}).
		\en
		
		{To see that $R(k,r)$ also has the same maximal eigenvalue, note that the
		 shifts of finite type determined by the matrices $M_3$ and $R(k,r)$ are topologically conjugate by state splitting or amalgamation \cite{LM}*{Theorem 2.4.10, p. 54} and hence have the same maximal eigenvalue.}
	 
		Alternatively, 
			for each $n \geq 0$ we can count on $L_n$ the number $u_n$ of ``free" vertices $g_{i_1}\dots g_{i_n}$ with $i_n \in \{1, \dots, k-r_k\}$ and the remaining number $v_n$ of ``restricted" vertices. Then
		\be
		\begin{bmatrix}
			u_0\\
			v_0
		\end{bmatrix}
		= 	\begin{bmatrix}
			1\\
			0
		\end{bmatrix}, \quad
		\begin{bmatrix}
			u_{n+1}\\
			v_{n+1}
		\end{bmatrix} =
		\begin{bmatrix}
			k-r_k & k-r_k\\
			r_k  &  0
		\end{bmatrix}
		\begin{bmatrix}
			u_{n}\\
			v_{n}
		\end{bmatrix} 
		\quad\text{ for } n \geq 0.
		\en
	The powers of these two matrices both count the number of vertices on the $n$'th row of $\Delta_n$, which has exponential growth rate given by the maximal eigenvalue. 
		\end{proof}
	
	Abbreviate $s=\sqrt{u^2+4uv}$. The recurrences in (\ref{eq:cderec}) have solutions
	\be\label{eq:sols}
	\begin{aligned}
		c(n)&=\frac{1}{2^{n+1}s} \left[ (u+s)(u-s)^n +(s-u)(u+s)^n\right],\\
		d(n)&=\frac{1}{2^ns}\left[ (u+s)^n - (u-s)^n \right],\\
		e(n)&= \frac{1}{s}\left[ (\frac{u+s}{2})^n-(\frac{u-s}{2})^n\right]=\frac{1}{s}(\beta^n-\alpha^n).
	\end{aligned}
	\en
	
	We also solve the recurrence for the number of vertices on row $n$ of the restricted tree: 
	\be\label{eq:L_n}
	\begin{aligned}
		|L_n|&=|R^{n-1}|=k|L_{n-1}|+v|L_{n-2}|, \quad L_0=1, L_1=k, L_2=ku+v, \dots \\
		|L_n|&=\frac{1}{2s}\left[ \beta^n (2k+s-u) -\alpha^n (2k-s-u)\right].
	\end{aligned}
	\en
	Thus
	\be
	|\Delta_n| \sim \frac{\beta}{\beta-1}|L_n| \sim \frac{\beta}{\beta-1} \frac{2k-u+s}{2s} \beta^n.
	\en
	
	\begin{theorem}\label{th:pressureexists}
		For a generalized Fibonacci tree, {assume that $\log g(n)$ is bounded (which is the case on the full $k$-tree, or if the matrix $A \in P(k,p)$ for some $p \geq 1$, or if $g(\alpha,\beta)>0$ on the closure of the orbit of $(\rho_1,\rho_0)$ under $T$)}. Then
	the limit below exists, and {the pressure on the restricted tree is}
		\be\label{eq:series1}
	P^{(k)}=	\lim_{n \to \infty} \frac{\log |Z_n|}{|\Delta_n|} =  \frac{\beta -1}{\beta(2k+s-u)} \left[ (s-u) \log |Z_0| + {2}\log |Z_1| +
		{2}\sum_{i=1}^\infty \frac{\log g(i)}{\beta^i}\right].
		\en
	\end{theorem}
	\begin{proof}
		Looking at (\ref{eq:Z_n}), we have
		\be
		\frac{\log |Z_n|}{|\Delta_n|}=\frac{c(n)}{|\Delta_n|} + \frac{d(n)}{|\Delta_n|} +\frac{1}{|\Delta_n|} \sum_{i=1}^{n-1} e(n-i) \log g(i).
		\en
		When in the first two terms we expand $c(n)$ and $d(n)$ using (\ref{eq:sols}), as $n \to \infty$ we can ignore the terms involving $(u-s)^n$, since $s+u>s-u>0$. 
		
		For the third term, 
		\be
		\frac{1}{|\Delta_n|} \sum_{i=1}^{n-1} e(n-i) \log g(i) = 
		\frac{1}{s|\Delta_n|} \sum_{i=1}^{n-1} (\beta^{n-i} - \alpha^{n-i}) \log g(i),
		\en
		split the sum into two series, one involving powers of $\beta$ and one involving powers of $\alpha$. 
		If $|\alpha|<1$, then the series involving powers of $\alpha$ converges, and when we divide by $|\Delta_n|$ and let $n \to \infty$ this part contributes nothing to the limit.
		
		If $|\alpha|\geq 1$, the sum in the third term involving powers of $\alpha$ is bounded by
		\be
		|\alpha|^n\sum_{i=1}^{n-1} \frac{\log g(i)}{\alpha^{i}} \leq |\alpha^n| (n-1) ||g||_\infty,
		\en
		so when we divide by $|\Delta_n|$, which is of order $\beta^n\gg\alpha^n$, again this contribution vanishes.
		
		Since {$\log g(i)$ is bounded and} as $n \to \infty$
		\be
		\begin{aligned}
			\frac{c(n)}{|\Delta_n|} &\sim \frac{\beta-1}{\beta}\frac{s-u}{2k-u+s},\\
			\frac{d(n)}{|\Delta_n|} &\sim 2 \frac{\beta-1}{\beta}\frac{s-u}{2k-u+s}, \quad\text{and}\\
			\frac{e(n)}{|\Delta_n|} &\sim 2 \frac{\beta-1}{\beta}\frac{s-u}{2k-u+s},
		\end{aligned}
		\en
		the conclusion follows.
		\end{proof}

	We specialize to the Fibonacci tree, when $k=2$ and $r=1$. Denote by $\lambda$ the golden mean, which is the maximal eigenvalue of the restriction matrix $R(2,1)$.
	\begin{corollary}
		For the Fibonacci tree and matrix $A \in P(k,p)$ for some $p \geq 1$, 
	the limit below exists, and
	\be\label{eq:series2}
	\lim_{n \to \infty} \frac{\log |Z_n|}{|\Delta_n|} =  \frac{1}{\lambda^5} \log |Z_0| + \frac{1}{\lambda^4} \log |Z_1| +
	\frac{1}{\lambda^4}\sum_{i=1}^{\infty} \frac{\log g(i)}{\lambda^i}.
	\en
	\end{corollary}

	\section{Asymptotic pressure on restricted trees}\label{sec:aspres}
		Suppose that for each dimension $k \geq 2$ we have a primitive $k \times k$ $0,1$ matrix $R_k$ that defines a subtree $S(k,R_k)$ of the full rooted $k$-tree as above (see \ref{eq:subtree}). 
	We denote by $\Delta_n^{(k)}$ the subtree of $S(k,R_k)$ consisting of all vertices of height no more than $n$ and by $\lambda_k$ the maximal (Perron-Frobenius) eigenvalue of $R_k$.

	\begin{lemma}\label{lem:sizes}
		For each $k>1$ the limit as $n \to \infty$ of the
		ratio of the size of the last row of $\Delta_n^{(k)}$ to the number of vertices in $\Delta_n^{(k)}$  is
		\be
	\lim_{n \to \infty} \frac{|L_n^{(k)}|}{|\Delta_n^{(k)}|} =	\lim_{n \to \infty} \frac{|R_k^{n-1}|}{1 + \sum_{m=1}^n|R_k^{m-1}|}=\frac{\lambda_k-1}{\lambda_k}.
		\en
		\end{lemma}
	\begin{proof}
		By the Perron-Frobenius Theorem (see for example \cite{Seneta 1981}*{Theorem 1.2, p. 9} and \cite{LM}*{Theorem 4.5.12, p.130}), if a nonnegative primitive matrix $M$ has maximal eigenvalue $\lambda$ and corresponding left and right eigenvectors $l$ and $r$, normalized so that $l \cdot r=1$, then for each  
		$i,j=1,\dots,k$ and $n \geq 0$ there are $\epsilon_{ij}(n)$ which tend to $0$ as $n \to \infty$ such that
		\be
		(M^n){ij} = \lambda^n(r_i l_j + \epsilon_{ij}(n)).
		\en 
			Thus for fixed $k$ (suppressing the dependence of $\epsilon$ on $k$), 
			and letting $\epsilon (n) =\sum_{i,j} \epsilon_{ij}(n)$,
		\be
		 \frac{|R_k^{n-1}|}{1 + \sum_{m=1}^n|R_k^{m-1}|}=\frac{\lambda_k^{n-1}(1 + \epsilon(n-1))}{1+\sum_{m=1}^n 
		 	\lambda_k^{m-1}(1 + \epsilon(m-1))}.
	 	\en
	 	The terms involving the $\epsilon(n)$ and $\epsilon(m)$ can be ignored as $n \to \infty$, because
	 	\be
	 	\frac{1}{\lambda_k^{n-1}} \sum_{m=1}^n \lambda_k^{m-1} \epsilon (m-1) \to 0 \quad\text{ as } n \to \infty.
	 		\en
	 	(To see this, given $\delta>0$, choose $n_0$ so that $n \geq n_0$ implies that $\epsilon(n)<\delta$. 
	 	Then choose $n_1 \geq n_0$ so that for $n \geq n_1$,
	 	\be
	 	\frac{1+\epsilon(1) \lambda_k +\epsilon(2) \lambda_k^2 + \dots + \epsilon(n_0-1) \lambda_k^{n_0-1}}{\lambda_k^{n-1}} < \delta.)
	 	\en
	 	This shows that for $n \geq n_1$,
	 	\be
	 	{ \frac{|R_k^{n-1}|}{1 + \sum_{m=1}^n|R_k^{m-1}|} \approx \frac{\lambda_k^{n-1}}{1+\sum_{m=1}^n \lambda_k^{m-1}} \approx \frac{\lambda_k-1}{\lambda_k}.}
	 	 \en	 	
	\end{proof}
	
	Suppose now that we
	have also a $d \times d$ nonnegative pair interaction matrix $A=(a_{ij})$ and a $d$-dimensional positive vector $w$.
	Recall that then the interaction matrix is $E(i,j)=a_{ij}w_j$.
	\begin{theorem}\label{thm:limit}
		Denote by $X_A^{(k)}$ the tree SFT on the subtree $S(k,r)$ determined by the primitive restriction matrix $R_k$ and nondegenerate pair interaction matrix $A$; by $|Z_n^{(k)}|$ its partition function on $\Delta_n^{(k)}$; by $P_n^{(k)}$ the pressure or free energy on $\Delta_n^{(k)}$; {and by $\bar P^{(k)}=\limsup_{n \to \infty} P_n^{(k)}$ and $\underline P^{(k)}=\liminf_{n \to \infty} P_n^{(k)}$ the upper and lower limiting pressures.}
		Denote the maximal eigenvalue of $R_k$ by $\lambda_k$ and assume that 
		\be
		\lim_{k \to \infty} \lambda_k = \infty.
		\en
		Denote by $s$ the maximum row sum of the interaction matrix $E$:
		\be
		s=\max\{\sum_{j=1}^d E(i,j):i=1,\dots,d\}.
		\en
		Then the {\em asymptotic pressure} of the sequence of restricted tree shifts (all with the same interaction matrix) as the dimension tends to infinity is
		\be
	P^{(\infty)}=	\lim_{k \to \infty} \bar P^{(k)} = \lim_{k \to \infty} \underline P^{(k)}=  \log s.
		\en
	\end{theorem}
\begin{proof}
	The argument at the end of \cite{PS2018} adapts to this generalized context. 
	Fix $k \geq 2$.
	When considering the possible values of $|Z_n^{(k)}|$ due to the various configurations $x \in X_A$, we see that the root $\epsilon$ can be assigned any of the values $w_1,\dots ,w_d$, and the edges entering any row with the same label $i$ at their initial vertices can independently be assigned  values $E(i,j)$ ($j=1, \dots, d$). 
	Each of these products is bounded by a power of the maximal row sum. 
	Thus
	\be
	|Z_n^{(k)}| \leq |w| s^{|L_1|} s^{|L_2|} \dots s^{|L_n|},
	\en
	so that 
	\be
	\frac{\log |Z_n^{(k)}|}{|\Delta_n|} \leq \frac{\log |w|+ \sum_{m=1}^n |L_m| \log s}{1+\sum_{m=1}^n |L_m|}\to \log s 
	\quad\text{ as } n \to \infty,
	\en
{and hence 
	\be
\bar P^{(k)} \leq \log s.
	\en }
	For the lower estimate, let $i \in \{1,\dots ,d\}$ be an index for which
	\be
	\sum_{j=1}^d E(i,j)=s,
	\en
	and consider configurations $x \in X_A$ for which $x(g)=i$ for all $g \in L_{n-1}$.
	
	We complete each configuration $x$ by next working up from row $L_{n-1}$ to the root. 
	On row $n-2$ assign at every vertex any allowed predecessor $j$ of $i$ (i.e., $a_{ji}>0$), 
	then on row $n-3$ assign at each vertex any allowed predecessor of $j$, etc.
	
	Finally, on row $n$ assign independently to each vertex any of the allowed successors of $i$. 
	This way we produce a large set $\mathcal C$ of allowed configurations on $\Delta_n$, each of which extends to a legal $x \in X_A$. 
	
{Let 
	\be
	b=\inf \{E(i,j): E(i,j) >0\} 
	\en
	and, for each $i,j$,
	\be
	\tilde E(i,j)=E(i,j)/b, \quad\text{ so that all } \tilde E(i,j)\geq 1.
	\en
	Let $\tilde s=s/b$ denote the maximum row sum of the matrix $\tilde E$	
	and $w_*$ the minimum entry of $w$. 
	Then
	\be
	\begin{aligned}
|Z_n^{(k)}| &\geq \sum_{x \in \mathcal C} w(x(\epsilon)) \prod_{\gamma=\left< g,h\right> \text{ edge in }\Delta_n} E(x(g),x(h))\\
	&= \sum_{x \in \mathcal C} w(x(\epsilon)) \prod_{\gamma=\left< g,h\right> \text{ edge in }\Delta_n} b \tilde E(x(g),x(h))\\
	&{\geq w_* b^{|\Delta_n|-1} \sum_{j_1,\dots ,j_{|L_n|}} \tilde E(i,j_1) \tilde E(i,j_2) \cdots \tilde E(i,j_{|L_n|})}\\
		&= w_* b^{|\Delta_n|-1} [\tilde E(i,1) + \tilde E(i,2) + \dots + \tilde E(i,d)]^{|L_n|}\\
			&= w_* b^{|\Delta_n|-1} \tilde s^{|L_n|}=w_* b^{|\Delta_n|-1-|L_n|} s^{|L_n|}.
	\end{aligned}
	\en }
{Therefore
	\be
	\frac{\log |Z_n^{(k)}|}{|\Delta_n|} \geq \frac{\log w_* +(|\Delta_n|-1-|L_n)\log b + |L_n| \log s}{|\Delta_n|},
	\en
	and	
	\be
	\begin{gathered}
	\lim_{n \to \infty} \frac{\log w_* +(|\Delta_n|-1-|L_n)\log b + |L_n| \log s}{|\Delta_n|}=
\frac{\log b}{\lambda_k} +	\frac{\lambda_k -1}{\lambda_k} \log s \\
\leq \liminf_{n \to \infty} \frac{\log |Z_n^{(k)}|}{|\Delta_n|}
	=\underline P^{(k)}\leq \bar P^{(k)} \leq \log s.
	\end{gathered}
		\en	
	}
{The conclusion follows by letting $k \to \infty$.}
\end{proof}
		
	\begin{corollary}\label{cor:irred}
		The previous theorem holds as well if the hypothesis that the restriction matrices $R_k$ be primitive is relaxed to require only that they be irreducible.
	\end{corollary}
\begin{proof}
	An irreducible nonnegative matrix $M$ has an associated strongly connected directed graph $G$ whose vertices are the indices of $M$. They share  a minimal period $p$.
	$M$ can be replaced by its canonical form produced by grouping indices into classes that are permuted cyclically 
	when following the walk that $M$, regarded as an adjacency matrix, defines on $G$.
	Then $M^p$ has square primitive matrices on its diagonal and $0$ entries elsewhere.	
	See \cite{LM}*{Sections 2.2 and 4.5} and \cite{Seneta1981}*{Section 1.3}.
	We apply this to the matrices $R_k$.
	
	As before denote by $\lambda_k$ the maximal eigenvalue of the irreducible nonnegative matrix $R_k$ and by $r,l$ the associated right and left positive eigenvectors. 
	To simplify the notation we suppress most of the dependence on $k$. 
	By \cite{Seneta1981}*{Theorems 1.3 and 1.4, pp. 18 and 21} and \cite{LM}*{Exercise 4.5.14, p. 134},
	for each $i,j=1,\dots,d$ there is a unique integer $t(i,j) \in [0,p-1]$ such that 
	\be
	(R^m)_{ij}>0 \quad\text{ implies } m \equiv t(i,j) \mod p, \text{ and, for all large enough } n,
	\en
	\be
     0 < (R^{np+t(i,j)})_{ij}=(r_il_j + \epsilon_{ij}(n)) \lambda_k^{np+t(i,j)},
	\en
	where each $\epsilon_{ij}(n)\to 0$ as $n \to \infty$.
	Thus the proof of Lemma \ref{lem:sizes} goes through as before.
	\end{proof}

\begin{example}\label{ex:GGM}
	Let us see how this works for the generalized Fibonacci trees defined by the restriction matrices $R(k,r_k)$ of Example \ref{ex:restriction matrices} (2).
	For each $k \geq 1$ and $r_k \in [0,k-1]$, assuming that $\log g$ is bounded, denote by $P^{(k)}$ the pressure on the (possibly) restricted tree shift determined by the tree restriction matrix $R_k=R(k,r_k)$ and the labeling matrix $A$ (the limit is known to exist by Theorem \ref{th:pressureexists}). 
	If $r_k=0$, we have the full $k$-tree.\\
	\indent
$R(k,r_k)$
has maximal eigenvalue $\lambda(k,r_k)=(k-r_k+\sqrt{(k-r_k)(k+3r_k)})/2$ (equal to the maximal eigenvalue of $M_3$, above). 
Each $\lambda(k,r) \to \infty$, in fact uniformly in $\{ (k,r_k): r_k \in [0,k-1]\}$. \\
\indent
Thus for such a sequence of (possibly) restricted trees, with fixed pair interaction matrix $A$, fixed site energy vector $w$ for labeling sites, and fixed interaction matrix $E(i,j)=a_{ij}w_j$ with maximal row sum $s$, we have that the asymptotic pressure is 
\be
\lim_{k \to \infty} P^{(k)} = \log s, \quad\text{ uniformly in } \{ (k,r_k): r_k \in [0,k-1]\}.
\en
\end{example}

 \end{document}